\documentclass[preprint]{elsarticle}

\usepackage[english]{babel}
\usepackage{enumerate}
\usepackage{xcolor}

\usepackage{amsmath}
\usepackage{amsfonts}
\usepackage{amsthm}
\usepackage{dsfont}
\usepackage{mathtools}
\usepackage{tikz}
\usetikzlibrary{shadows}

\usepackage{hyperref}
\hypersetup{
    colorlinks=true,
    linkcolor=blue!75!black,
    urlcolor=gray,
    linktoc=all
}

\xdefinecolor{olive}{cmyk}{0.64,0,0.95,0.4}
\xdefinecolor{blueberry}{cmyk}{0.44,0.23,0.0,0.18}
\xdefinecolor{maroon}{cmyk}{0.00,1.0,1.0,0.498}

\newcommand{\ignore}[1]{}

\newcommand{\RR}{\mathds{R}}
\newcommand{\ZZ}{\mathds{Z}}

\newcommand{\QQ}{\mathds{Q}}

\newcommand{\ii}{\mathrm{i}}

\newcommand{\0}{\mathbf{0}}
\newcommand{\1}{\mathbf{1}}

\newcommand{\mat}[1]{\begin{bmatrix}#1\end{bmatrix}}

\newcommand{\widet}[1]{\widetilde{#1}}

\newtheorem{thm}{Theorem} 
\newtheorem{lem}[thm]{Lemma} 
\newtheorem{corollary}[thm]{Corollary} 
 
\newdefinition{rmk}{Remark} 
\newdefinition{definition}{Definition}
\newproof{pf}{Proof}
\newproof{pot}{Proof of Theorem \ref{thm2}}

\DeclarePairedDelimiter{\alg}{\langle}{\rangle}

\DeclareMathOperator{\spn}{span}

\title{Laplacian pretty good fractional revival}
\author[1]{Ada Chan}

\author[2]{Bobae Johnson}

\author[3]{Mengzhen Liu}

\author[4]{Malena Schmidt}

\author[5]{Zhanghan Yin}

\author[1]{Hanmeng Zhan}

\address[1]{Department of Mathematics and Statistics, York University, Toronto, Canada}
\address[2]{Departments of Mathematics and Physics, Harvard University, Cambridge, United States}
\address[3]{Department of Pure Mathematics and Mathematical Statistics, University of Cambridge, Cambridge, United Kingdom}
\address[4]{Department of Computer Science, University of Warwick, Coventry, United Kingdom}
\address[5]{Department of Mathematics, University of Toronto, Toronto, Canada}

\begin{document}

\begin{abstract}
We develop the theory of pretty good fractional revival in quantum walks on graphs using their Laplacian matrices as the Hamiltonian.   
We classify the paths and the double stars that have Laplacian pretty good fractional revival.
\end{abstract}

\begin{keyword}
Continuous-time quantum walk, pretty good fractional revival, Laplacian matrix, paths, double stars
\end{keyword}

\maketitle

\section{Introduction}
\label{Section:Intro}

The continuous-time quantum walk on a graph is a time dependent evolution, given by the Schr\"odinger equation using a Hermitian Hamiltonian associated with the graph. In 2003, Childs et al. gave a quantum walk based algorithm that solves an oracular problem exponentially faster than any classical algorithm \cite{MR2121062}. Continuous-time quantum walks can also be viewed as a universal primitive for quantum computation \cite{MR2507892}.

The continuous-time quantum walk on a graph $X$ is given by the transition matrix $e^{-\ii t H}$  where $H$ is a Hamiltonian associated with $X$.
Two commonly used Hamiltonians are the adjacency matrix and the Laplacian matrix of $X$.
In this paper, we introduce the notion of pretty good fractional revival in the quantum walk on $X$ with its Laplacian matrix  being 
the Hamiltonian.    For pretty good fractional revival in quantum walks using the adjacency matrix as its Hamiltonian, please see \cite{MR4357783}.

Let $X$ be a connected graph and $L$ be the Laplacian matrix of $X$.
The transition matrix of the continuous-time quantum walk on $X$ is
\begin{equation*}
U(t) = e^{-\ii t L}.
\end{equation*}

Let $a$ and $b$ be vertices in $X$, we use $e_a$ and $e_b$ to denote their characteristic vectors.
We say the walk has {\em Laplacian fractional revival from $a$ to $b$ at time $\tau$} if 
\begin{equation*}
U(\tau) e_a = \alpha e_a+\beta e_b 
\end{equation*}
for some complex numbers $\alpha$ and $\beta$ satisfying $|\alpha|^2+|\beta|^2=1$.   
When $\beta=0$, we say that the walk is {\sl periodic} at $a$ at time $\tau$.   
When $\alpha=0$, we say the walk admits {\em Laplacian perfect state transfer from $a$ to $b$}.  

Perfect state transfer and fractional revival are useful for quantum information transport and entanglement generation \cite{Bose2003, CDDEKL2005, CDEL2004}.
It is known that graphs with perfect state transfer are rare \cite{MR2992400}.  For instance, Coutinho and Liu prove that the path of length two is the only tree that admits Laplacian perfect state transfer \cite{MR3421609}.
This led to a relaxation, called pretty good state transfer, introduced by Godsil et al. \cite{GKSSPGST}, and Vinet et al. \cite{VZAlmost}.
A graph $X$ has {\em Laplacian pretty good state transfer from $a$ to $b$} if for all $\epsilon>0$, there exists a time $t_{\epsilon}>0$ such that $|U(t_{\epsilon})_{a,b}|>1-\epsilon$.
For paths, Banchi et al. \cite{MR3627144} determined all the paths that have pretty good state transfer between extremal vertices, and van Bommel \cite{vanBommelThesis} subsequently extends to the following result to cover internal vertices.
\begin{thm}
\label{Thm:LaPGST}
Laplacian pretty good state transfer occurs between vertices $a$ and $b$ in a path of length $n$ if and only if $n$ is a power of two and $a+b=n+1$.
\end{thm}

The situation of Laplacian fractional revival  is only marginally better,  the authors show in \cite{MR4287701} that the path of length three is the only other tree admitting Laplacian fractional revival. 
 In this paper, we present a relaxation of Laplacian fractional revival, called {\em Laplacian pretty good fractional revival}, and give the following classification.
\begin{thm}
\label{Thm:Main1}
Laplacian pretty good fractional revival occurs between vertices $a$ and $n+1-a$ in a path of length $n$ if and only if one of the following holds.
\begin{enumerate}[i.]
\item
$n=p^{\ell}$ for some prime $p$, integer $\ell\geq 1$, and $a\neq \frac{p^{\ell}+1}{2}$.
\item
$n=2p^{\ell}$ for some odd prime $p$, integer ${\ell} \geq 1$, and $a=\frac{p^{\ell}+1}{2}$ or $\frac{3p^{\ell}+1}{2}$.
\end{enumerate}
\end{thm}

A major difference from the adjacency analogue of this result,  Theorem 1.1 of  \cite{MR4357783}, is that Laplacian pretty good fractional revival can occur only between symmetric vertices ($a$ and $n+1-a$) on a path where as the path of length $5\cdot 2^k-1$ has (adjacency) pretty good fractional revival between vertices $2^k$ and $3\cdot 2^k$.
This difference stems from the fact that the Laplacian matrix always have $\bf{1}$ as an eigenvector.  As shown in the proof of Lemma~\ref{Lem:StrCop}, if a Hamiltonian has an eigenvector 
$\bf{u}$ where $\bf{u}_a=\bf{u}_b \neq 0$, then pretty good fractional revival between $a$ and $b$ implies that they are strongly cospectral vertices.   If a Hamiltonian, such as the adjacency matrix of a path, does not have such eigenvectors, then pretty good fractional revival implies a weaker necessary condition on the vertices  called strongly fractional cospectrality \cite{MR4357783}.

For quantum walks with adjacency Hamiltonian,  we get number-theoretic conditions on the length of paths where pretty good state transfer or pretty good fractional revival occur
\cite{GKSSPGST,MR4357783}. 
In \cite{MR3005296},  Godsil et al. present the double stars $S(k,k)$ as another class of graphs where (adjacency) pretty good state transfer occurs if and only if a number-theoretic condition holds.
In the last section of this paper, we show that for the Laplacian Hamiltonian, there is no number theoretic restriction on double stars admitting Laplacian pretty good state transfer  or Laplacian pretty good fractional revival.

\begin{thm}
Laplacian pretty good fractional revival occurs in the double star $S(n,m)$ if and and only if one of the following holds.
\begin{enumerate}
\item
$n=m$, Laplacian pretty good state transfer  occurs between the two non-pendant vertices;
\item
$n\neq m$ and $n=2$, Laplacian pretty good fractional revival occurs between the two pendant neighbours of the vertex of degree three;
\item
$n=m=1$, $P_4$ has Laplacian pretty good state transfer  between the extremal vertices.
\end{enumerate}
\end{thm}


\section{Laplacian pretty good fractional revival}
\label{Section:Laplacian pretty good fractional revival}

Let $X$ be a connected graph with Laplacian matrix $L$ that has spectral decomposition
$L = \sum_{r=0}^d \mu_r E_r$.
Then the transition matrix
\begin{equation*}
U(t) = e^{-\ii t L}= \sum_{r=0}^d e^{-\ii t \mu_r} E_r.
\end{equation*}
The Laplacian matrix $L$ is positive semidefinite and $L\1 = 0 \1$, where $\1$ is the all $1$'s vector. 

As $\| U(t) \| \leq \sum_{r=0}^d \| E_r \|$, the set
\begin{equation*}
\Gamma_X = \{ U(t)  \ :\   t \in \RR\}
\end{equation*}
is bounded and its closure, $\overline{\Gamma}_X$, is compact. 
As $I \in\overline{\Gamma}_X$, Lemma~6.1 in \cite{MR3005296} implies that the continuous-time quantum walk on any graph is approximately periodic at every vertex.

\begin{definition}
The graph $X$ has {\em Laplacian pretty good fractional revival} between vertices $a$ and $b$ if $\overline{\Gamma}_X$ contains a limit point with the block diagonal form 
\begin{equation*}
\begin{bmatrix} N & \0\\ \0 & N' \end{bmatrix}
\end{equation*}
for some $2\times 2$ unitary matrix $N$ indexed by $\{a,b\}$.
\end{definition}
If $N$ is diagonal then $X$ is almost periodic at both $a$ and $b$.  
We have Laplacian pretty good state transfer  between $a$ and $b$ in $X$ if
$N$ has zero diagonal entries.
Note that Laplacian fractional revival  occurs if $\Gamma_X$ contains a matrix with the above block diagonal form.

 In this article, we focus on Laplacian pretty good fractional revival that is not approximately periodic at both $a$ and $b$, that is, $N$ is not diagonal.  
 We call this phenomenon {\em proper} Laplacian pretty good fractional revival.

Suppose $X$ is a connected graph admitting Laplacian pretty good fractional revival between two vertices with
limit point $M=\mat{N & \0\\ \0 & N'}$ where $N$ is not diagonal.

Consider a sequence $\{U(t_k)\}_{k \geq 1}$ that converges to $M$.
For $r=0,\ldots, d$, an eigenvector $\phi$ in the $r$-th eigenspace of $L$ satisfies
\begin{equation*}
U(t_k) \phi = e^{-\ii t_k \mu_r} \phi.
\end{equation*}
Using $\eta_r$ to denote the subsequential limit of $\{e^{-\ii t_k \mu_r}\}_{k \geq 1}$, we have 
\begin{equation*}
\mat{N & \0\\ \0 & N'} \phi = \eta_r \phi.
\end{equation*}
Hence we can write
\begin{equation}
\label{Eqn:NBlkDecomp}
\mat{N & \0\\ \0 & N'}  = \sum_{r=0}^d \eta_r E_r
\end{equation}
and it belongs to the algebra $\alg{L}$ consisting of all polynomials in $L$. 

For any square matrix $M$ and vector $\phi$ indexed by $V(X)$, we use $\widet{M}$ and $\widet{\phi}$ to denote the restriction of $M$ and $\phi$ to $\{a,b\}$.
We use $I_n$ and $J_n$ to denote the identity matrix of order $n$ and the $n\times n$ matrix of all ones, respectively.

\begin{lem}
\label{Lem:StrCop}
If $\alg{L}$ contains
\begin{equation*}
\mat{N & \0\\ \0 & N'}
\end{equation*}
for some non-diagonal matrix $N$ indexed by $\{a,b\}$ then
\begin{equation*}
E_r e_a = \pm E_r e_b, \qquad \text{for $r=0,\ldots, d$.}
\end{equation*}
\end{lem}
\begin{proof}
First observe that $N$ is symmetric, and for any eigenvector $\phi$ of $L$,  $\widet{\phi}$ is either the zero vector or an eigenvector of $N$.
In particular, $\mat{1&1}^T$ is an eigenvector of $N$.  Hence $N$ has constant row sum, and $N \in \spn\{I_2, J_2\}$.

Therefore
\begin{equation*}
N = \mat{\alpha & \beta\\ \beta & \alpha},
\end{equation*}
for some $\alpha, \beta$ where $\beta\neq 0$.   Then $N$ has two distinct eigenvalues $\alpha+\beta$ and $\alpha-\beta$.
It follows from (\ref{Eqn:NBlkDecomp}) that either $\widet{E_r} = \0$ or $N \widet{E_r} = \eta_r \widet{E_r}$, for $r=0,\ldots, d$.

Since $E_r$ is a symmetric idempotent, we have $(E_r)_{a,a} = \| E_r e_a \|^2$.  Hence $\widet{E_r}=\0$ implies $E_re_a = E_r e_b =\0$.

Now suppose $\widet{E_r} \neq \0$, then $\eta_r = \alpha \pm \beta$.  It follows from (\ref{Eqn:NBlkDecomp}) that
\begin{equation*}
E_r \mat{N & \0\\ \0 & N'}  = \eta_r E_r.
\end{equation*}
Restricting to the first column yields  $\alpha E_re_a+\beta E_r e_b = \eta_r E_r e_a$.   The result follows from $\beta\neq 0$.
\end{proof}

\begin{definition}
The vertices $a$ and $b$ in $X$ are {\em strongly cospectral with respect to $L$} if 
\begin{equation*}
E_r e_a = \pm E_r e_b
\end{equation*} 
for $r=0,\ldots, d$.
\end{definition}

\begin{corollary}
\label{Cor:SC}
If proper Laplacian pretty good fractional revival occurs between $a$ and $b$ in $X$, then $a$ and $b$ are strongly cospectral with respect to $L$.
\hfill \qed
\end{corollary}

As in the study of Laplacian pretty good state transfer,  Kronecker's theorem \cite{LevitanZhikov} is a key element in the study of Laplacian pretty good fractional revival.
\begin{thm}[Kronecker]
\label{Thm:Kronecker}
Let $\theta_0, \ldots, \theta_d$ and $\zeta_0,\ldots, \zeta_d$ be real numbers.  For arbitrary small positive $\epsilon$, the system of inequalities
\begin{equation*}
|\theta_r y - \zeta_r| < \epsilon \pmod{2\pi}, \qquad \text{for $r=0,\ldots,d$}
\end{equation*}
has a solution $y$ if and only if, for integers $l_0,\ldots,l_d$,
\begin{equation*}
l_0\theta_0 + \ldots + l_d \theta_d =0
\end{equation*}
implies
\begin{equation*}
l_0 \zeta_0+\ldots+l_d \zeta_d =0 \pmod{2\pi}.
\end{equation*}
\qed
\end{thm}

Let  $a$ and $b$ be strongly cospectral vertices with respect to $L=\sum_r \mu_r E_r$.  We define 
\begin{eqnarray}
\label{Eqn:Partition}
\Phi_{a,b}^0 &=& \{ \mu_r \ :\  E_re_a=E_re_b=\0\},  \nonumber\\
\Phi_{a,b}^+ &=& \{ \mu_r  \ :\   E_r e_a = E_r e_b \neq \0\} \qquad \text{and}\\
\Phi_{a,b}^- &=& \{ \mu_r  \ :\   E_r e_a = -E_r e_b \neq \0\}. \nonumber
\end{eqnarray}

Using the  proof of Theorem 2.4 in \cite{MR4357783}, we get the following characterization of graphs having Laplacian pretty good fractional revival.
\begin{thm}
\label{Thm:Laplacian pretty good fractional revival}
Let $X$ be a connected graph with vertices $a$ and $b$.
There is proper Laplacian pretty good fractional revival between $a$ and $b$ if and only if the following conditions hold.
\begin{enumerate}
\item
$a$ and $b$ are strongly cospectral with respect to $L$, and
\item
for all integers $l_0, \ldots,l_d$,
\begin{equation*}
\sum_{\mu_j \in \Phi_{a,b}^+}l_j \mu_j + \sum_{\mu_j \in \Phi_{a,b}^-}l_j\mu_j = 0
\end{equation*}
implies
\begin{equation*}
\sum_{\mu_j \in \Phi_{a,b}^-}l_j \neq \pm 1.
\end{equation*}
\qed
\end{enumerate}
\end{thm}

\section{Paths}
\label{Section:Paths}

In this section, we prove Theorem~\ref{Thm:Main1} which classifies the paths that admit proper Laplacian pretty good fractional revival.   Our result extends Theorem~\ref{Thm:LaPGST}, the classification for pretty good state transfer on paths \cite{vanBommelThesis}.

Let $P_n$ be the path on $n$ vertices $\{1, \ldots, n\}$ with vertex $a$ adjacent to $a-1$ and $a+1$.  The Laplacian matrix $L$ of $P_n$ has  simple eigenvalues
$\mu_0=0$ and 
\begin{equation}
\label{Eqn:PathEigenvalues}
\mu_r = 2+2\cos\left(\frac{r \pi}{n}\right), \quad \text{for $r=1,\ldots, n-1$.}
\end{equation}
The eigenvector corresponding to $\mu_0$ is $\phi_0 = \1$, and for $r=1,\ldots, n-1$,
the eigenvector corresponding to $\mu_r$ is defined by
\begin{equation*}
\phi_r(j)= (-1)^j \sin \left(\frac{(2j-1)r \pi}{2n} \right)\cos \left(\frac{r \pi}{2n}\right), \quad \text{for $j=1,\ldots, n$.}
\end{equation*}

It follows immediately that $a$ is strongly cospectral to $b$ if and only if $b=n+1-a$. 
When $n$ is even, we have
\begin{eqnarray}
\Phi_{ab}^+  &=& \{\mu_r : r \text{ is even  and } 2n \not \vert (2a-1)r\} \cup \{0\} \quad \text{and}
\label{Eqn:SupportEven}\\
\Phi_{ab}^-  &=& \{\mu_r : r \text{ is odd}\}.
\nonumber
\end{eqnarray}
When $n$ is odd,  we have
\begin{eqnarray}
\Phi_{ab}^+   &=& \{\mu_r : r \text{ is odd}\}  \cup \{0\} \quad \text{and}
\label{Eqn:SupportOdd}\\
\Phi_{ab}^- &=& \{\mu_r : r \text{ is even  and } 2n \not \vert (2a-1)r\}.
\nonumber
\end{eqnarray}

We first show that proper Laplacian pretty good fractional revival occurs in the paths given in Theorem~\ref{Thm:Main1}.
\begin{lem}
\label{Lem:Path1}
The path $P_n$ admits proper Laplacian pretty good fractional revival between $a$ and $n+1-a$ if one of the following holds.
\begin{enumerate}[i.]
\item
$n=p^{\ell}$, for some prime $p$, $\ell \geq 1$, and $a\neq \frac{p^{\ell}+1}{2}$,
\item
$n=2p^{\ell}$, for some odd prime $p$, $\ell \geq 1$, and $a \in \left\{\frac{p^{\ell}+1}{2}, \frac{3p^{\ell}+1}{2} \right\}$.
\end{enumerate}

\end{lem}
\begin{proof}
When $n=2^{\ell}$, the result follows from Theorem~\ref{Thm:LaPGST}.

Let $p$ be an odd prime and $b=n+1-a$.
Suppose $l_0, \ldots, l_{n-1}$ are integers satisfying 
$l_r=0$, for $\mu_r\in \Phi_{ab}^0$, and
\begin{equation}
\label{Eqn:Path1}
 \sum_{j= 0}^{n-1}  l_j\mu_j = \sum_{j= 1}^{n-1}  l_j\mu_j =0.
\end{equation}
Let $\omega = e^{\frac{2\pi \ii}{2n}}$.  Then $\mu_j=2+\omega^j+\omega^{2n-j}$.
It follows from Equation~(\ref{Eqn:Path1}) that $\omega$ is a root of the polynomial
\begin{equation*}
L(x)= 2\sum_{j=1}^{n-1} l_j + \sum_{j=1}^{n-1} l_j x^j + \sum_{j=1}^{n-1}l_j x^{2n-j}.
\end{equation*}
Let $\Psi_{2n}(x)$ be the $2n$-th cyclotomic polynomial. Then there exists an integer polynomial $g(x)$ satisfying
\begin{equation*}
L(x) = \Psi_{2n}(x) g(x).
\end{equation*}

When $n=p^{\ell}$, $\Psi_{2p^\ell}(-1) = \Psi_{p^\ell}(1)=p$.  It follows from Equation~(\ref{Eqn:SupportOdd})  that
\begin{equation*}
L(-1) = 4\sum_{j \ \text{even}} l_j =4\sum_{\mu_j \in \Phi_{ab}^-} l_j.
\end{equation*}
Hence 
\begin{equation*}
4\sum_{\mu_j \in \Phi_{ab}^-} l_j=  pg(-1),
\end{equation*}
and $p$ divides $\sum_{\mu_j \in \Phi_{ab}^-} l_j$.  By Theorem~\ref{Thm:Laplacian pretty good fractional revival}, $P_{p^k}$ has Laplacian pretty good fractional revival between $a$ and $n+1-a$.

When $n=2p^{\ell}$ and $p^{\ell}$ is a factor of $(2a-1)$, Equation~(\ref{Eqn:SupportEven}) gives
\begin{equation*}
\Phi_{ab}^+ = \{\mu_r \ :\ r \equiv 2 \pmod{4}\} \cup \{0\}.
\end{equation*}
As a result,
\begin{equation*}
L(\ii) = 2\sum_{\mu_r \in \Phi_{ab}^-} l_j = \Psi_{2n}(\ii) g(\ii) = p g(\ii).
\end{equation*}
Now $g(\ii) = (2\sum_{\mu_r \in \Phi_{ab}^-} l_j)/p$ is an algebraic integer, so it is an integer.  We conclude that $\sum_{\mu_r \in \Phi_{ab}^-} l_j  \neq \pm 1$ and, by Theorem~\ref{Thm:Laplacian pretty good fractional revival}, $P_{2p^k}$ has Laplacian pretty good fractional revival between $a$ and $n+1-a$.

\end{proof}
Note that when $n=p^{\ell}$ and $a=\frac{p^{\ell}+1}{2}$, $n+1-a=a$ is not strongly cospectral to any other vertex in $P_n$.

We need the following results to rule out all other cases.
\begin{lem}
\label{Lem:PathId1}
Let $n=km$ for some odd integer $m \geq 3$.  Then,
\begin{equation*}
\label{Eqn:PathId1a}
\sum_{j=1}^{m-1} (-1)^j\mu_{kj} =-2.
\end{equation*}
and, for $1\leq s\leq k-1$,
\begin{equation*}
\label{Eqn:PathId1b}
\sum_{j=0}^{m-1} (-1)^j\mu_{kj+s} =2.
\end{equation*}
\end{lem}
\begin{proof}
The equations follow immediately from
Lemma~5 of \cite{MR3981791} and Equation~(\ref{Eqn:PathEigenvalues}).
\end{proof}

\begin{lem}
\label{Lem:PathId2}
Let $n=km$ for some odd integer $m \geq 3$. We have
\begin{equation*}
\label{Eqn:PathId2}
\sum_{j=0}^{\frac{m-3}{2}}\mu_{(2j+1)k} =m.
\end{equation*}
\end{lem}
\begin{proof}
First observe that
\begin{equation*}
 \sum_{j=0}^{m-1} e^{\frac{(2j+1)\pi\ii}{m}}=e^{\frac{2\pi\ii}{m}} \left( \sum_{j=0}^{m-1} e^{\frac{2\pi\ii}{m}j}\right)=0.
\end{equation*}
The real part  gives
\begin{eqnarray*}
\sum_{j=0}^{m-1} \cos\frac{(2j+1)\pi}{m} = \sum_{j=0}^{\frac{m-3}{2}}2 \cos \frac{(2j+1)\pi}{m} - 1 =0.
\end{eqnarray*}
Equation~(\ref{Eqn:PathId2}) follows from Equation~(\ref{Eqn:PathEigenvalues}).
\end{proof}

\begin{lem}
\label{Lem:Path2}
Let $n=2p$ for some odd prime $p$.
There is no proper Laplacian pretty good fractional revival in $P_n$ between vertices $a$ and $n+1-a$ if $p$ does not divide $(2a-1)$.
\end{lem}
\begin{proof}
Let $a$ be a vertex such that $p$  does not divide $2a-1$, and let $b=n+1-a$.

Now $\mu_p=2$, and applying the first equation in Lemma~\ref{Lem:PathId1} with $m=p$ and $k=2$ gives
\begin{equation*}
\sum_{j=1}^{p-1}(-1)^j\mu_{2j}+\mu_{p}=0.
\end{equation*}
Since $p$ does not divide $(2a-1)$ , we have $\Phi_{ab}^+=\{0, \mu_2, \mu_4, \ldots, \mu_{2p-2}\}$.
For $r=1,\ldots, n-1$, let $l_r$ be the coefficient of $\mu_r$ in the above equation, and let $l_0=0$.
Then we have
\begin{equation*}
\sum_{\mu_r \in \Phi_{ab}^+} l_r \mu_r + \sum_{\mu_r \in \Phi_{ab}^-} l_r \mu_r = 0
\end{equation*}
but
\begin{equation*}
\sum_{\mu_r \in \Phi_{ab}^-} l_r = l_p = 1.
\end{equation*}
By Theorem~\ref{Thm:Laplacian pretty good fractional revival}, proper Laplacian pretty good fractional revival does not occur between $a$ and $n+1-a$ in $P_{2p}$
if $p$ does not divide $2a-1$.
\end{proof}

\begin{lem}
\label{Lem:Path3}
Let $n=2p^{\ell}$ for some odd prime $p$ and $\ell \geq 2$.
There is no proper Laplacian pretty good fractional revival in $P_n$ between vertices $a$ and $n+1-a$ if $p^{\ell}$ does not divide $(2a-1)$.
\end{lem}
\begin{proof}
Let $a$ be a vertex such that $p^{\ell}$  does not divide $2a-1$, and let $b=n+1-a$.

First $\mu_{p^{\ell}}=2$.  Applying Lemma~\ref{Lem:PathId1} with $m=p$, $k=2p^{\ell-1}$ and  $s= 2$ gives
 \begin{equation}
 \label{Eqn:Path3}
\sum_{j=0}^{p-1}(-1)^j\mu_{2p^{\ell-1} j+2}- \mu_{p^{\ell}}=0.
\end{equation}
Since $p$ does not divide $2p^{\ell-1} j+2$, we see that $4p^{\ell} \not\vert (2a-1)(2p^{\ell-1} j+2)$ and
\begin{equation*}
2p^{\ell-1}j+2 \in \Phi_{ab}^+, \quad \text{for $j=0,\ldots, p-1$.}
\end{equation*}
For $r=1,\ldots, n-1$, let $l_r$ be the coefficient of $\mu_r$ in Equation~(\ref{Eqn:Path3}), and let $l_0=0$.
Then we have
\begin{equation*}
\sum_{\mu_r \in \Phi_{ab}^+} l_r \mu_r + \sum_{\mu_r \in \Phi_{ab}^-} l_r \mu_r = 0
\end{equation*}
but
\begin{equation*}
\sum_{\mu_r \in \Phi_{ab}^-} l_r = l_{p^{\ell}} = -1.
\end{equation*}
By Theorem~\ref{Thm:Laplacian pretty good fractional revival}, proper Laplacian pretty good fractional revival does not occur between $a$ and $n+1-a$ in $P_{2p^{\ell}}$
if $p^{\ell}$ does not divide $2a-1$.
 \end{proof}

\begin{lem}
\label{Lem:Path4}
Proper Laplacian pretty good fractional revival does not occur in $P_n$ if $n=2hq$ for some odd integers $h, q \geq 3$ satisfying $\gcd(h,q)=1$.
\end{lem}
\begin{proof}
Let $a$ be any vertex in $P_n$ and $b=n+1-a$.

Applying Lemma~\ref{Lem:PathId2} with $(m, k)=(q, 2h)$ and $(m, k)=(h, 2q)$ gives
\begin{equation*}
\sum_{j=0}^{\frac{q-3}{2}}\mu_{2h(2j+1)} =q
\end{equation*}
and
\begin{equation*}
\sum_{j=0}^{\frac{h-3}{2}}\mu_{2q(2j+1)} =h
\end{equation*}
respectively.
Since $4$ does not divide $2h(2j+1)(2a+1)$ nor $2q(2j+1)(2a+1)$ for any integer $j$, we see that the eigenvalues in the above equations belong to $\Phi_{ab}^+$.

As $\gcd(h, q)=1$, there exist integers $s$ and $t$ such that $sq+th=1$.
Together with $\mu_{hq} = 2$, we get
\begin{equation*}
\mu_{hq} - 2s \left(\sum_{j=0}^{\frac{q-3}{2}}\mu_{2h(2j+1)}\right) -2t \left(\sum_{j=0}^{\frac{h-3}{2}}\mu_{2q(2j+1)}\right)=0.
\end{equation*}

For $r=1, \ldots, n-1$, let $l_r$ be the coefficient of $\mu_r$ in the above equation and let $l_0=0$.  Then
\begin{equation*}
\sum_{\mu_r \in \Phi_{ab}^+} l_r \mu_r + \sum_{\mu_r \in \Phi_{ab}^-} l_r \mu_r = 0
\end{equation*}
but
\begin{equation*}
\sum_{\mu_r \in \Phi_{ab}^-} l_r = l_{hq} = 1.
\end{equation*}
By Theorem~\ref{Thm:Laplacian pretty good fractional revival}, proper Laplacian pretty good fractional revival does not occur  in $P_{2hq}$.

\end{proof}

\begin{lem}
\label{Lem:Path5}
Proper Laplacian pretty good fractional revival does not occur in $P_n$ if  $n=2^{\ell} m $, for some odd integer $m\geq 3$ and integer $\ell\geq 2$.
\end{lem}
\begin{proof}
Let $a$ be any vertex in $P_n$ and $b=n+1-a$.

First $\mu_{2^{\ell-1}m}=2$.  Applying Lemma~\ref{Lem:PathId1} with , $k=2^{\ell}$ and  $s\in \{1,2\}$ gives
 \begin{equation}
 \label{Eqn:Path5}
\sum_{j=0}^{m-1}(-1)^j\mu_{2^{\ell} j+1}- \sum_{j=0}^{m-1}(-1)^j\mu_{2^{\ell} j+2}=0.
\end{equation}
Since $4$ does not divide $(2^{\ell} j+2)(2a-1)$, we see that, for $j=0,\ldots, p-1$,
\begin{equation*}
2^{\ell} j+2 \in \Phi_{ab}^+.
\end{equation*}
For $r=1,\ldots, n-1$, let $l_r$ be the coefficient of $\mu_r$ in Equation~(\ref{Eqn:Path5}), and let $l_0=0$.
Then we have
\begin{equation*}
\sum_{\mu_r \in \Phi_{ab}^+} l_r \mu_r + \sum_{\mu_r \in \Phi_{ab}^-} l_r \mu_r = 0
\end{equation*}
but
\begin{equation*}
\sum_{\mu_r \in \Phi_{ab}^-} l_r = \sum_{j=0}^{m-1}(-1)^j = 1.
\end{equation*}
By Theorem~\ref{Thm:Laplacian pretty good fractional revival}, proper Laplacian pretty good fractional revival does not occur  $P_{2^{\ell}m}$.
\end{proof}

\begin{lem}
\label{Lem:Path6}
Proper Laplacian pretty good fractional revival does not occur in $P_n$ when $n$ is odd with at least two distinct prime factors.
\end{lem}
\begin{proof}
Let $n=p_1^{f_1}p_2^{f_2}\cdots p_h^{f_h}$ where $h\geq 2$ and $p_1, \ldots, p_h$ are distinct odd primes.
For each vertex $a$, it is strongly cospectral with only $b=n+1-a$ so we can assume $a<\frac{n+1}{2}$.

Let $\gcd(n, 2a-1)=p_1^{e_1}p_2^{e_2}\cdots p_h^{e_h}$.   Without loss of generality, we assume $e_1<f_1$.
Applying Lemma~\ref{Lem:PathId1} with $m=p_2^{f_2}\cdots p_h^{f_h}$, $k=p_1^{f_1}$ and $s \in \{1,2\}$ gives
\begin{equation}
\label{Eqn:Path6a}
\sum_{j=0}^{p_2^{f_2}\cdots p_h^{f_h}-1} (-1)^j \mu_{p_1^{f_1}j+1} + \sum_{j=0}^{p_2^{f_2}\cdots p_h^{f_h}-1} (-1)^j \mu_{p_1^{f_1}j+2}=4.
\end{equation}
For $s \in \{1, 2\}$,  $p_1$ does not divide $p_1^{f_1}j+s$, and 
\begin{equation*}
p_1^{f_1}j+s \in 
\begin{cases} 
\Phi_{ab}^-  & \text{if  $p_1^{f_1}j+s$ is even,}\\
\Phi_{ab}^+  & \text{if  $p_1^{f_1}j+s$ is odd.}
\end{cases}
\end{equation*}

We apply Lemma~\ref{Lem:PathId2} with $m=p_1^{f_1}$ and $m=p_h^{f_h}$ to get
\begin{equation}
\label{Eqn:Path6b}
\sum_{j=0}^{\frac{p_1^{f_1}-3}{2}}\mu_{(2j+1)p_2^{f_2}\cdots p_h^{f_h}} = p_1^{f_1},
\end{equation}
and
\begin{equation}
\label{Eqn:Path6c}
\sum_{j=0}^{\frac{p_h^{f_h}-3}{2}}\mu_{(2j+1)p_1^{f_1}\cdots p_{h-1}^{f_{h-1}}} = p_h^{f_h},
\end{equation}
respectively.  Note that the eigenvalues in these two equations belong to $\Phi_{ab}^+$.

Let $s$ and $t$ be integers such that $s p_1^{f_1} + t p_h^{f_h}=1$.  Then from Equations~(\ref{Eqn:Path6a}) to (\ref{Eqn:Path6c}), we get
\begin{equation*}
\sum_{j=0}^{p_2^{f_2}\cdots p_h^{f_h}-1} (-1)^j \left(\mu_{p_1^{f_1}j+1}+\mu_{p_1^{f_1}j+2}\right)
- 4s \left(\sum_{j=0}^{\frac{p_1^{f_1}-3}{2}}\mu_{(2j+1)p_2^{f_2}\cdots p_h^{f_h}}\right) - 4t \left(\sum_{j=0}^{\frac{p_h^{f_h}-3}{2}}\mu_{(2j+1)p_1^{f_1}\cdots p_{h-1}^{f_{h-1}}}\right)=0.
\end{equation*}
For $r=1, \ldots, n-1$, let $l_r$ be the coefficient of $\mu_r$ in the above equation and let $l_0=0$.
\begin{equation*}
\sum_{\mu_r \in \Phi_{ab}^+} l_r \mu_r + \sum_{\mu_r \in \Phi_{ab}^-} l_r \mu_r = 0
\end{equation*}
but
\begin{equation*}
\sum_{\mu_r \in \Phi_{ab}^-} l_r =\sum_{j=0}^{p_2^{f_2}\cdots p_h^{f_h}-1} (-1)^j = 1.
\end{equation*}
By Theorem~\ref{Thm:Laplacian pretty good fractional revival}, proper Laplacian pretty good fractional revival does not occur  in $P_{n}$.
\end{proof}

Combining the above lemmas, we reach the following classification of paths that admit proper Laplacian pretty good fractional revival.
\begin{thm}
\label{Thm:Path}
Laplacian pretty good fractional revival occurs between vertices $a$ and $n+1-a$ in a path of length $n$ if and only if one of the following holds.
\begin{enumerate}[i.]
\item
$n=p^{\ell}$ for some prime $p$ and integer $\ell\geq 1$ and $a\neq \frac{p^{\ell}+1}{2}$.
\item
$n=2p^{\ell}$ for some odd prime $p$ and integer ${\ell} \geq 1$, and $a=\frac{p^{\ell}+1}{2}$ or $\frac{3p^{\ell}+1}{2}$.
\end{enumerate}
\end{thm}


\section{Double stars}
\label{Section:DS}

For $m,n \geq 1$, the double star $S(m,n)$ is constructed by attaching $m$ pendant vertices to one vertex of $K_2$ and attaching $n$ pendant vertices to the other vertex
of $K_2$.

In \cite{MR3005296}, Fan and Godsil rule out (adjacency) perfect state transfer in $S(m,n)$, and they show that (adjacency) pretty good state transfer occurs between the non-pendant vertices in $S(m,m)$ if and only if $4m+1$ is not a perfect square.   
Brown et al. \cite{GST} show that (adjacency) fractional revival, (and hence adjacency pretty good fractional revival), occurs in $S(m,m)$, for $m\geq 1$.  
Laplacian perfect state transfer and Laplacian fractional revival in trees with more than three vertices are ruled out in \cite{MR3421609} and \cite{MR4287701}.
In this section, we classify the double stars that
admit Laplacian pretty good fractional revival.  In particular, we show that $S(m,m)$ has Laplacian pretty good state transfer  (and proper Laplacian pretty good fractional revival) between the two non-pendant vertices, for $m\geq 1$.

\begin{lem}
If proper Laplacian pretty good fractional revival occurs between $a$ and $b$ in $S(m,n)$ then one of the following holds.
\begin{enumerate}
\item
$m=n$, $a$ and $b$ are non-pendant vertices;
\item
$n=2$, $a$ and $b$ are the pendant vertices adjacent to the vertex of degree three;
\item
$m=n=1$, $a$ and $b$ are the extremal vertices of $P_4$.
\end{enumerate}
\end{lem}
\begin{proof}
Suppose proper Laplacian pretty good fractional revival occurs between $a$ and $b$ in $S(m,n)$.
By Corollary~\ref{Cor:SC}, $a$ and $b$ are strongly cospectral with respect to $L$, and $(E_r)_{a,a}=\| E_re_a\|^2 = (E_r)_{b,b}$, for all $r$.   
It follows from the spectral decomposition of $L$ that $a$ and $b$ have the same degree, that is, either both $a$ and $b$ have degree 1 or
both $a$ and $b$ have degree $m+1$ in the case where $n=m$.

Since proper Laplacian pretty good fractional revival occurs between $a$ and $b$, $\alg{L}$ contains a block diagonal matrix
\begin{equation*}
M=\mat{N & \0\\ \0 & N'}
\end{equation*}
for some non-diagonal matrix $N$ indexed by $\{a,b\}$.  
Let $P$ be a permutation matrix representing an automorphism  $\Theta$ of $S(m,n)$ fixing $a$, then $PM=MP$ which implies that $\Theta$ also fixes $b$.
Therefore if $a$ and $b$ are pendant vertices then  either $a$ and $b$ are the extremal vertices of $S(1,1)$ or they are adjacent to the same non-pendant vertex that has degree exactly three.  We have, without loss of generality, $n=2$.
\end{proof}

Let $L$ be the Laplacian matrix of the balanced double star $S(m,m)$.   It is straightforward to compute the spectrum of $L$ and check that for non-pendant vertices $a$ and $b$, 
we have $\Phi_{ab}^0 = \{1\}$, $\Phi_{ab}^+ = \{0, m+1\}$ and
\begin{equation*}
\Phi_{ab}^- = \left\{\frac{m+3+\sqrt{m^2+6m+1}}{2}, \frac{m+3-\sqrt{m^2+6m+1}}{2}\right\}.
\end{equation*}

Before we show that $S(m,m)$ has Laplacian pretty good state transfer  between the non-pendant vertices, we need to following characterization from \cite{MR3627144}.
\begin{thm}
\label{Thm:Laplacian pretty good state transfer Char}
Laplacian pretty good state transfer  occurs between $a$ and $b$ if and only if whenever there are integers $l_j$'s satisfying
\begin{equation*}
\sum_{\mu_j \in \Phi_{ab}^+} l_j \mu_j + \sum_{\mu_j \in \Phi_{ab}^-} l_j \mu_j =0
\end{equation*}
then
\begin{equation*}
\sum_{\mu_j \in \Phi_{ab}^-} l_j  \quad \text{is even.}
\end{equation*}
\end{thm}

\begin{lem}
For $m\geq 1$, $m^2+6m+1$ is not a perfect square. 
\end{lem}
\begin{proof}
Suppose $m^2+6m+1=k^2$ for some $k\geq 1$.  Then
\begin{equation*}
(m+k+3)(m-k+3)=8=1\times 8=2\times 4,
\end{equation*}
but the difference between the factors are $(m+k+3)-(m-k+3)=2k$ is even, so $2k=4-2=2$ which implies  $m=0\text{ or }-6$. 
\end{proof}

\begin{lem}
\label{Lem:S(m,m)}
For $m\geq 1$, Laplacian pretty good state transfer  occurs between the non-pendant vertices $a$ and $b$ in the balanced double star $S(m,m)$.
\end{lem}
\begin{proof}
Let $\mu_0=0$, $\mu_1=m+1$, $\mu_2=\frac{m+3+\sqrt{m^2+6m+1}}{2}$ and $\mu_3= \frac{m+3-\sqrt{m^2+6m+1}}{2}$.
Suppose 
\begin{equation*}
l_0 (0) + l_1(m+1) + l_2 \left(\frac{m+3+\sqrt{m^2+6m+1}}{2}\right) + l_3\left(\frac{m+3-\sqrt{m^2+6m+1}}{2}\right) =0.
\end{equation*}
Since $m^2+6m+1$ is not a perfect square, the above equation holds if and only if $l_2=l_3$.   Hence the sum $\sum_{\mu_j \in \Phi_{ab}^-} l_j  =2l_3$ and 
$S(m,m)$ has Laplacian pretty good state transfer  between the two non-pendant vertices.
\end{proof}

It follows from Theorem~\ref{Thm:LaPGST} that $P_4$ has Laplacian pretty good state transfer  between its extremal vertices.   The only case left is $S(m,2)$ 
where $a$ and $b$ are the pendant vertices adjacent to the vertex of degree three.
The Laplacian matrix of $S(m,2)$ is
\begin{equation*}
L=\mat{ 1& 0&-1&0 &0&0&\cdots&0\\0&1&-1&-1&0&0&\cdots&0\\-1&-1&3&-1& 0&0&\cdots&0\\0&0&-1&m+1& -1&-1&\cdots&-1\\0&0&0&-1&1&0&\cdots&0\\
0&0&0&-1&0&1&\cdots&0\\0&0&0&\vdots&0&0&\ddots&0\\0&0&0&-1&0&0&\cdots&1},
\end{equation*}
where the first two rows and columns are indexed by $\{a, b\}$, 
 and the third and fourth rows and columns are indexed by the two non-pendant vertices.

The eigenvalues of $L$ are $0$, $1$ and the roots of the cubic polynomial 
\begin{equation*}
p_m(x)=x^3-(m+6)x^2+(4m+9)x-(m+4).    
\end{equation*}  
The eigenvectors for eigenvalue $1$ lie in
\begin{equation*}
\spn\big\{\mat{1 & -1& 0 & 0& \0}^T, \mat{0&0&0&0&u}^T \ :\  \text{$u$ is orthogonal to $\1_m$} \big\}.
\end{equation*}
For each root $\theta$ of $p_m(x)$, the eigenvector is 
\begin{equation*}
\mat{1&1&(1-\theta)&(\theta^2-4\theta+1)& \frac{(\theta^2-4\theta+1)}{1-\theta} \1_m}^T.
\end{equation*}

\begin{lem}
For $m\geq 1$, $p_m(x)$ is reducible over $\ZZ$ if and only if $m=2$.
\end{lem}
\begin{proof}
Consider $d(x)=p_{m+1}(x) - p_m(x) = - x^2 +4x - 1$, which is independent of $m$.   The roots of $d(x)$ are $2\pm \sqrt{3}$.
Observe that 
\begin{equation*}
p_m(2\pm \sqrt{3}) = p_1(2\pm \sqrt{3}), \quad\text{for $m\geq 1$.}
\end{equation*} 

Suppose $m\geq 3$.  
Now $p_m(\frac{1}{2}) = \frac{6m - 7}{8} > 0$ and
\begin{equation*}
p_m(2-\sqrt{3}) = p_1(2-\sqrt{3}) = -2 < 0. 
\end{equation*}
By Intermediate Value Theorem, $p_m(x)$ has a root $\theta_1$ in the interval
$ (2-\sqrt{3},\frac{1}{2})$.

Similarly, $p_m(3) = 2m - 4 > 0$, and  $p_m(2+\sqrt{3}) = p_1(2+\sqrt{3}) = -2 < 0$.  Applying Intermediate Value Theorem again,  $p_m(x)$ has
a root $\theta_2$ in the interval $(3 , 2+\sqrt{3})$.

Note that $\theta_1, \theta_2 \not \in \ZZ$ and $0<\theta_1\theta_2 < \frac{1}{2} (4)$.   Let $\theta_3$ be the third root of $p_m(x)$.  Then $\theta_1\theta_2\theta_3=m+4$.
If $\theta_3$ is an integer then the algebraic integer $\theta_1\theta_2 \in \ZZ$.  We conclude that $\theta_1\theta_2=1$ and $\theta_3=m+4$.  But $p_m(m+4)=2m(m+4)\neq 0$. 
As a result, $\theta_3$ is not an integer and $p_m(x)$ is irreducible for $m\geq 3$.

When $m=2$, $p_2(x) = (x-3)(x^2-5x+2)$.
When $m=1$, $p_1(x) = x^3 - 7x^2 + 13x - 5$ is irreducible. 
\end{proof}

\begin{lem}
\label{Lem:S(m,2)}
For $m\neq 2$, $S(m,2)$ admits Laplacian pretty good fractional revival between the two pendant vertices adjacent to the vertex of degree three.
\end{lem}

\begin{proof}
Let $a,b$ denote the two pendant vertices adjacent to the vertex of degree three.  Let $\theta_1$, $\theta_2$ and $\theta_3$ be the roots of $p_m(x)$.
Then $\Phi_{ab}^-=\{1\}$ and $\Phi_{ab}^+ = \{0, \theta_1, \theta_2, \theta_3\}$.

Let $K$ be the splitting field of $p_m(x)$ over $\QQ$. There exists a cyclic subgroup $Gal(K/\QQ)$ of order $3$ that permutes the roots of $p_m(x)$ cyclically
\cite{MR2286236}.
For any integers $r, l_1, l_2, l_3$ satisfying
\begin{equation*}
r \cdot 1 + l_1 \theta_1+l_2 \theta_2 +l_3 \theta_3=0,
\end{equation*}
we also have
$r \cdot 1 + l_1 \theta_2+l_2 \theta_3 +l_3 \theta_1=0$ and 
$r \cdot 1 + l_1 \theta_3+l_2 \theta_1 +l_3 \theta_2=0$.
Adding the three equations gives
\begin{equation*}
 r=-\frac{(l_1 + l_2 + l_3)(\theta_1 + \theta_2 + \theta_3)}{3} =-\frac{(l_1 + l_2 + l_3)(m+6)}{3} \neq \pm 1.
\end{equation*}
By Theorem~\ref{Thm:Laplacian pretty good fractional revival}, Laplacian pretty good fractional revival occurs between $a$ and $b$.
\end{proof}

\begin{lem}
\label{Lem:S(2,2)}
There is no Laplacian pretty good fractional revival between the two pendant vertices adjacent to the vertex of degree three in $S(2,2)$.
\end{lem}
\begin{proof}
Let $a,b$ denote the two pendant vertices adjacent to the vertex of degree three. 
Then $\Phi_{ab}^+ =\{0, 3, \frac{5+\sqrt{17}}{2},\frac{5-\sqrt{17}}{2}\}$ and $\Phi_{ab}^-=\{1\}$.
If we let $r=1$, $l_1=3$, $l_2=l_3=-2$ we have
\begin{equation*}
r (1) + l_1 (3) + l_2  \left( \frac{5+\sqrt{17}}{2}\right)+l_3 \left(\frac{5-\sqrt{17}}{2}\right)= 0
\end{equation*}
and $\sum_{\mu_j \in \Phi_{ab}^-} l_j=r=1$.   By Theorem~\ref{Thm:Laplacian pretty good fractional revival}, Laplacian pretty good fractional revival does not occur between $a$ and $b$.
\end{proof}

\begin{thm}
Laplacian pretty good fractional revival occurs in the double star $S(n,m)$ if and and only if one of the following holds.
\begin{enumerate}
\item
$n=m$, Laplacian pretty good state transfer  occurs between the two non-pendant vertices;
\item
$n\neq m$ and $n=2$, Laplacian pretty good fractional revival occurs between the two pendant neighbours of the vertex of degree three;
\item
$n=m=1$, $P_4$ has Laplacian pretty good state transfer  between the extremal vertices.
\end{enumerate}
\end{thm}
\begin{proof}
The theorem follows from Theorem~\ref{Thm:LaPGST} and Lemmas~\ref{Lem:S(m,m)}, \ref{Lem:S(m,2)} and \ref{Lem:S(2,2)}.
\end{proof}

\section{Acknowledgement}
This project was supported by the Fields Undergraduate Summer Research Program.  Chan acknowledges the support of NSERC Discovery Grant RGPIN-2021-03609.  Zhan acknowledges the support of the York Science Fellow program.
The authors would like to thank Gabriel Coutinho, Chris Godsil, Chayim Lowen and Christino Tamon for useful discussions.


\bibliographystyle{elsarticle-num}
\bibliography{LaPGFR}

\begin{thebibliography}{10}
\expandafter\ifx\csname url\endcsname\relax
  \def\url#1{\texttt{#1}}\fi
\expandafter\ifx\csname urlprefix\endcsname\relax\def\urlprefix{URL }\fi
\expandafter\ifx\csname href\endcsname\relax
  \def\href#1#2{#2} \def\path#1{#1}\fi

\bibitem{MR2121062}
A.~M. Childs, R.~Cleve, E.~Deotto, E.~Farhi, S.~Gutmann, D.~A. Spielman,
  \href{https://doi-org.ezproxy.library.yorku.ca/10.1145/780542.780552}{Exponential
  algorithmic speedup by a quantum walk}, in: Proceedings of the
  {T}hirty-{F}ifth {A}nnual {ACM} {S}ymposium on {T}heory of {C}omputing, ACM,
  New York, 2003, pp. 59--68.
\newblock \href {http://dx.doi.org/10.1145/780542.780552}
  {\path{doi:10.1145/780542.780552}}.
\newline\urlprefix\url{https://doi-org.ezproxy.library.yorku.ca/10.1145/780542.780552}

\bibitem{MR2507892}
A.~M. Childs,
  \href{https://doi-org.ezproxy.library.yorku.ca/10.1103/PhysRevLett.102.180501}{Universal
  computation by quantum walk}, Phys. Rev. Lett. 102~(18) (2009) 180501, 4.
\newblock \href {http://dx.doi.org/10.1103/PhysRevLett.102.180501}
  {\path{doi:10.1103/PhysRevLett.102.180501}}.
\newline\urlprefix\url{https://doi-org.ezproxy.library.yorku.ca/10.1103/PhysRevLett.102.180501}

\bibitem{MR4357783}
A.~Chan, W.~Drazen, O.~Eisenberg, M.~Kempton, G.~Lippner,
  \href{https://doi-org.ezproxy.library.yorku.ca/10.5802/alco}{Pretty good
  quantum fractional revival in paths and cycles}, Algebr. Comb. 4~(6) (2021)
  989--1004.
\newblock \href {http://dx.doi.org/10.5802/alco} {\path{doi:10.5802/alco}}.
\newline\urlprefix\url{https://doi-org.ezproxy.library.yorku.ca/10.5802/alco}

\bibitem{Bose2003}
S.~Bose, Quantum communication through an unmodulated spin chain, Phys. Rev.
  Lett. 91~(20) (2003) 207901.

\bibitem{CDDEKL2005}
M.~Christandl, N.~Datta, T.~Dorlas, A.~Ekert, A.~Kay, A.~Landahl, Perfect
  transfer of arbitrary states in quantum spin network, Phys. Rev. A 71~(3)
  (2005) 032312.

\bibitem{CDEL2004}
M.~Christandl, N.~Datta, A.~Ekert, A.~Landahl, Perfect state transfer in
  quantum spin network, Phys. Rev. Lett. 92~(18) (2004) 187902.

\bibitem{MR2992400}
C.~Godsil,
  \href{https://doi-org.ezproxy.library.yorku.ca/10.13001/1081-3810.1563}{When
  can perfect state transfer occur?}, Electron. J. Linear Algebra 23 (2012)
  877--890.
\newblock \href {http://dx.doi.org/10.13001/1081-3810.1563}
  {\path{doi:10.13001/1081-3810.1563}}.
\newline\urlprefix\url{https://doi-org.ezproxy.library.yorku.ca/10.13001/1081-3810.1563}

\bibitem{MR3421609}
G.~Coutinho, H.~Liu,
  \href{https://doi-org.ezproxy.library.yorku.ca/10.1137/140989510}{No
  {L}aplacian perfect state transfer in trees}, SIAM J. Discrete Math. 29~(4)
  (2015) 2179--2188.
\newblock \href {http://dx.doi.org/10.1137/140989510}
  {\path{doi:10.1137/140989510}}.
\newline\urlprefix\url{https://doi-org.ezproxy.library.yorku.ca/10.1137/140989510}

\bibitem{GKSSPGST}
C.~D. Godsil, S.~Kirkland, S.~Severini, J.~Smith,
  \href{http://journals.aps.org/prl/abstract/10.1103/PhysRevLett.109.050502}{{Number-theoretic
  nature of communication in quantum spin systems.}}, Physical review letters
  109~(5) (2012) 050502.
\newblock \href {http://dx.doi.org/10.1103/PhysRevLett.109.050502}
  {\path{doi:10.1103/PhysRevLett.109.050502}}.
\newline\urlprefix\url{http://journals.aps.org/prl/abstract/10.1103/PhysRevLett.109.050502}

\bibitem{VZAlmost}
L.~Vinet, A.~Zhedanov,
  \href{http://link.aps.org/doi/10.1103/PhysRevA.86.052319}{{Almost perfect
  state transfer in quantum spin chains}}, Physical Review A 86~(5) (2012)
  052319.
\newblock \href {http://dx.doi.org/10.1103/PhysRevA.86.052319}
  {\path{doi:10.1103/PhysRevA.86.052319}}.
\newline\urlprefix\url{http://link.aps.org/doi/10.1103/PhysRevA.86.052319}

\bibitem{MR3627144}
L.~Banchi, G.~Coutinho, C.~Godsil, S.~Severini,
  \href{https://doi-org.ezproxy.library.yorku.ca/10.1063/1.4978327}{Pretty good
  state transfer in qubit chains---the {H}eisenberg {H}amiltonian}, J. Math.
  Phys. 58~(3) (2017) 032202, 9.
\newblock \href {http://dx.doi.org/10.1063/1.4978327}
  {\path{doi:10.1063/1.4978327}}.
\newline\urlprefix\url{https://doi-org.ezproxy.library.yorku.ca/10.1063/1.4978327}

\bibitem{vanBommelThesis}
C.~van Bommel, {Quantum walks and pretty good state transfer on paths}, Ph.{D}.
  thesis, Waterloo (2019).

\bibitem{MR4287701}
A.~Chan, B.~Johnson, M.~Liu, M.~Schmidt, Z.~Yin, H.~Zhan,
  \href{https://doi-org.ezproxy.library.yorku.ca/10.37236/10146}{Laplacian
  fractional revival on graphs}, Electron. J. Combin. 28~(3) (2021) Paper No.
  3.22, 28.
\newblock \href {http://dx.doi.org/10.37236/10146} {\path{doi:10.37236/10146}}.
\newline\urlprefix\url{https://doi-org.ezproxy.library.yorku.ca/10.37236/10146}

\bibitem{MR3005296}
X.~Fan, C.~Godsil,
  \href{https://doi-org.ezproxy.library.yorku.ca/10.1016/j.laa.2012.10.006}{Pretty
  good state transfer on double stars}, Linear Algebra Appl. 438~(5) (2013)
  2346--2358.
\newblock \href {http://dx.doi.org/10.1016/j.laa.2012.10.006}
  {\path{doi:10.1016/j.laa.2012.10.006}}.
\newline\urlprefix\url{https://doi-org.ezproxy.library.yorku.ca/10.1016/j.laa.2012.10.006}

\bibitem{LevitanZhikov}
B.~Levitan, Z.~V.V., Almost Periodic Functions and Differential Equations,
  Cambridge University Press, 1983.

\bibitem{MR3981791}
C.~M. van Bommel, A complete characterization of pretty good state transfer on
  paths, Quantum Inf. Comput. 19~(7-8) (2019) 601--608.

\bibitem{GST}
L.~C. Brown, W.~J. Martin, D.~Wright,
  \href{https://arxiv.org/pdf/2103.08837.pdf}{Continuous time quantum walks on
  graphs: group state transfer} (2021).
\newline\urlprefix\url{https://arxiv.org/pdf/2103.08837.pdf}

\bibitem{MR2286236}
D.~S. Dummit, R.~M. Foote, Abstract algebra, 3rd Edition, John Wiley \& Sons,
  Inc., Hoboken, NJ, 2004.

\end{thebibliography}

\end{document}